\documentclass{amsart}
\usepackage{amsmath,amssymb,amsthm}
\usepackage[margin=1.3in]{geometry}

\newtheorem{thm}{Theorem}
\newtheorem{cor}{Corollary}

\newcommand{\cS}{\mathcal S}
\newcommand{\cDS}{\mathcal D^\cS}
\newcommand{\cH}{\mathcal H}
\newcommand{\cM}{\mathcal M}
\newcommand{\WWKL}{\mathsf{WWKL}}

\title{Weakly 2-randoms and 1-generics in Scott sets}
\author[L.\ B.\ Westrick]{Linda Brown Westrick}
\address{Department of Mathematics\\
University of Connecticut\\
Storrs, Connecticut U.S.A.}
\email{westrick@uconn.edu}

\begin{document}

\begin{abstract}
  Let $\mathcal S$ be a Scott set, or even an
  $\omega$-model of $\mathsf{WWKL}$. Then 
for each $A\in \cS$, either there is $X \in \cS$ that is
  weakly 2-random relative to $A$, or 
there is $X\in \cS$ that is 1-generic relative to $A$.
It follows that 
if $A_1,\dots, A_n \in \cS$ are non-computable,
  there is $X \in \mathcal S$ such that each $A_i$ is
  Turing incomparable with $X$, answering
a question
  of Ku\v{c}era and Slaman.  
More generally,
any $\forall\exists$ sentence in the language of 
partial orders that holds in 
$\mathcal D$ also holds in $\cDS$,
where $\cDS$ is the partial order of Turing degrees 
of elements of $\cS$.
\end{abstract}

\maketitle

Ku\v{c}era and Slaman \cite{kuceraslaman} 
showed that for every Scott set $\cS$ and every
non-computable $A \in \cS$, there is $X \in \cS$ such that $X$ and $A$
are Turing incomparable.
Their non-uniform proof hinged
on whether $A$ was a $K$-trivial set.  
Conidis \cite{conidis} generalized this result to require only that
$\cS$ is an $\omega$-model
of $\WWKL$  (meaning that
for every $A\in \cS$, every $\Pi^0_1(A)$ class
\emph{of positive measure} has an element in $\cS$).
Ku\v{c}era and Slaman left open whether,
given non-computable $A, B \in \cS$, one could find $X \in \cS$
incomparable with both.  We answer that question in the affirmative, 
generalizing both of the above results 
(Corollary \ref{thm4}).

This is related to the question, asked by Ku\v{c}era and Slaman, 
of which $\forall\exists$ sentences 
(in the language of partial orders) hold in $\cDS$, the partial 
order induced by $\leq_T$ on the degrees of
$\cS$.  
One direction of the 
proof of the decidability of the
$\forall\exists$ theory of the Turing degrees $\mathcal D$ 
rests on a technical result (see \cite[Theorem II.4.11]{lerman}) 
which we generalize from $2^\omega$
to an arbitrary Scott set $\cS$ in Theorem \ref{thm5}.  
Whenever Theorem \ref{thm5} holds for some $\cS$, then every 
$\forall\exists$ sentence that holds in $\mathcal D$ also 
holds in $\cDS$ (Corollary \ref{cor6}).  
Li and Slaman \cite{liwei} recently showed that
whenever $\cS\subseteq \Delta^0_2$ is a Scott set and $A\in \cS$, 
then there is $G \in \cS$ such that $G$ is 1-generic relative to $A$. 
Theorem \ref{thm5} and Corollary \ref{cor6}
for the case where
$\cS \subseteq \Delta^0_2$ follow immediately.

Theorem \ref{thm3} is the new result which allows the usual proof 
of Theorem \ref{thm5} for $\cS = 2^\omega$ to lift to arbitrary 
$\omega$-models of $\WWKL$.
Corollary \ref{thm4} is also a
special case of Theorem \ref{thm5}.

The author would like to thank Uri Andrews, Joe Miller, Noah Schweber
and Mariya Soskova, with whom some partial results
were obtained by a different method
during the author's visit to
UW-Madison.  This note would not exist without the benefit 
of those discussions.

To prove Theorem \ref{thm3}, we
relativize the following theorem characterizing the 
weakly 2-random ($W2R$) sets in the Martin-L\"of random ($MLR$) sets. 
The relativizations of these notions to a set $A$ are denoted by
$W2R^A$ and $MLR^A$ respectively.

\begin{thm}[Hirschfeldt and Miller, \mbox{see \cite[Thm 5.3.16]{niesbook}}; Downey, Nies, Weber and Yu \cite{dnwy}]\label{hm}
  An $MLR$ set is $W2R$ if and only if it does not compute any
  non-computable c.e. set.
\end{thm}

\begin{thm}\label{thm3}
  For any Scott set $\cS$ and any $A \in \cS$, either there is $X \in \cS$
  that is weakly 2-random relative to $A$, or there is $X \in \cS$ that
  is 1-generic relative to $A$.
\end{thm}
\begin{proof}
  Let $\mathcal P$ be a $\Pi^0_1(A)$ class consisting entirely of $MLR^A$ reals.
  Let $X \in \cS \cap \mathcal P$, and suppose that $X$ is not $W2R^A$.  By the
  relativization of Theorem \ref{hm} to $A$, there is a set $D\leq_T X\oplus A$
  such that $D$ is c.e. in $A$ and $D \not\leq_T A$.  
Relativizing the proof that every 
non-computable c.e. set computes a 1-generic, 
there is a $G \leq_T D \oplus A$ such that $G$ is 1-generic
  relative to $A$.
\end{proof}
Because $\mathcal P$ has positive measure, Theorem \ref{thm3} 
and all the other results here hold even if
$\cS$ is only an $\omega$-model of $\WWKL$.
If every element of $\cS$ is $\Delta^0_2$, then no element is $W2R$,
and if every element of $\cS$ is hyperimmune-free, then no element is
1-generic, so the result is sharp.

\begin{cor}\label{thm4}
  For any Scott set $\cS$ and any non-computable
  $A_1,\dots, A_n \in \cS$, there is $X \in \cS$
  such that $X$ is Turing incomparable with each $A_i$.
\end{cor}
\begin{proof}
  Let $A = \bigoplus_{i\leq n} A_i$.
  If some $X \in \cS$ is 1-generic relative to $A$,
  then $X$ is 1-generic relative to each $A_i$, which suffices.
  So suppose some $X\in \cS$ is $W2R^A$. Then for each $i$, 
$X \not\geq_T A_i$, because $\{ Y : Y \geq_T A_i\}$ 
is a $\Sigma^0_3(A_i)$ set of measure 0.
\end{proof}

The same method gives a more general extension of embeddings result, 
entirely analogous to the extension of embeddings result that holds 
in $\mathcal D$.  The only addition is the observation that 
a weakly 2-random can play the same role as a 1-generic in the 
usual proof of this theorem in $\mathcal D$.

If $\cH = (H,\leq_H), \cM = (M,\leq_M)$ are partial orders, 
then $\cH \prec \cM$ means that $H\subseteq M$ and $\leq_H$ and $\leq_M$ 
agree on elements of $H$.  We say that 
$\cM$ is an \emph{end-extension} of $\cH$ if for every $h \in H$ and 
$m \in M \setminus H$, we have $m \not\leq_M h$.  If $\cH$ is an 
upper semi-lattice with join $\vee_H$, we say that $\cM$ 
\emph{respects the joins of $\cH$} if for every $h,k \in H$ and 
$m \in M\setminus H$, if $m \geq_M h$ and $m\geq_M k$, then 
$m \geq_M h \vee_H k$.

\begin{thm}\label{thm5}
Let $\cH \prec \cM$ be finite partial orders, where 
$\cH$ is an upper semi-lattice.  
Let $\cS$ be a Scott set, and let
$f:H\rightarrow \cS$ be an upper semi-lattice embedding of 
$\cH$ into $(\cS, \leq_T, \oplus)$.
Suppose $\cM$ is an end-extension of $\cH$ 
which respects the joins of $\cH$.  Then $f$ can be extended to  
$\hat f:M\rightarrow \cS$ which is a partial order embedding
of $\cM$ into $(\cS, \leq_T)$.
\end{thm}
\begin{proof}  For $h \in H$, let $A_h = f(h)$.
Let $G \in \mathcal S$ be either 1-generic or weakly 2-random relative to 
$\bigoplus_{h \in H} A_h$.  Interpret $G$ as the join of finitely many 
columns, one for each $m \in M \setminus H$, so that the reals $G_m$ 
are defined by $G = \bigoplus_{m \in M\setminus H} G_m$. 
For each $m \in M \setminus H$, define 
$$A_m = \bigoplus_{\substack{h \in H \\ h \leq_M m}} A_h 
\oplus \bigoplus_{\substack{\ell \in M\setminus H \\ \ell \leq_M m}} G_\ell.$$
We claim that $\hat f(m) := A_m$ satisfies 
the conclusion of the theorem.  If $G$ is 1-generic, 
we have just repeated the usual construction.  If $G$ is 
weakly 2-random, there is only one verification step that is different:
we claim that if $h \in H$ and $m \in M\setminus H$ with $h \not\leq_M m$, 
then $A_h \not\leq A_m$.  Let 
$$\mathcal B = \{Z : \bigoplus_{\substack{k \in H\\ k \leq_M m}}A_k \oplus Z \geq_T A_h\}$$
If $\mathcal B$ had positive measure, 
a majority vote argument would provide a computation to show
$\bigoplus_{\substack{k \in H\\ k \leq_M m}}A_k \geq_T A_h$.  This is not 
possible because $m\not\geq_M h$ and $\cM$ respects the joins of $\cH$.  
Therefore $\mathcal B$ is null, and it 
is $\Sigma^0_3(\bigoplus_{k \in H} A_k)$, so 
$G \not\in \mathcal B$, so $A_m \not\in \mathcal B$.

All the other steps in the verification are the same after the observation 
that because $G$ is $MLR^{\bigoplus_{k\in H} A_k}$, we have
$\bigoplus_{h\in H} A_h \oplus \bigoplus_{\ell \neq m} G_\ell \not\geq_T G_m$ 
for all $m \in M \setminus H$.
\end{proof}

We can now give a partial answer to the question asked in \cite{kuceraslaman} 
of whether the $\forall\exists$ theory (in the language of partial orders)
of the degrees of a Scott set 
coincides with the $\forall\exists$ theory of $\mathcal D$. 

\begin{cor}\label{cor6}
For any Scott set $\cS$,
any $\forall\exists$ sentence that holds in 
$\mathcal D$ also holds in $\cDS$.
\end{cor}
\begin{proof}
An \emph{extension sentence} is a sentence 
of the form $\forall \bar x \exists \bar y 
(\theta(\bar x) \rightarrow \bigvee_{j<n} \eta_j(\bar x, \bar y))$, 
where $\theta(\bar x)$ and each $\eta_j(\bar x , \bar y)$ are
conjunctions of atomic formulas and negated atomic formulas 
describing the complete $\leq$-diagram on their inputs, where
the diagram described by 
$\theta$ is an upper semi-lattice, and where each $\eta_j$ 
describes a diagram extending the diagram given by $\theta$.
It is known (see \cite[Theorem VII.4.4]{lerman}) that for any
$\forall\exists$ sentence $\psi$, there is a finite conjunction of 
extension sentences $\phi$, such that in any upper semi-lattice, 
$\psi$ holds if and only if all the conjuncts hold.
Furthermore, in $\mathcal D$, an extension sentence 
$\phi$ holds if and only if there is a $j$ such that the diagram $\cM_j$
described by $\eta_j$ is an end extension of the diagram $\cH$
described by $\theta$, and $\cM$ respects the joins of $\cH$ 
(see \cite[Theorems II.4.11, VII.4.1]{lerman}).

Applying Theorem \ref{thm5}, 
if an extension sentence $\phi$ holds in $\mathcal D$, 
then $\phi$ also holds in $\cDS$.  
Therefore, if $\psi$ holds in $\mathcal D$, then all the conjuncts 
$\phi$ hold in both structures,
so $\psi$ holds in $\cDS$ as well.
\end{proof}

It remains open whether all Scott sets satisfy exactly these $\forall\exists$ 
sentences.

\bibliographystyle{plain}
\bibliography{scottbib.bib}

\begin{thebibliography}{1}

\bibitem{conidis}
Chris~J. Conidis.
\newblock A measure-theoretic proof of {T}uring incomparability.
\newblock {\em Ann. Pure Appl. Logic}, 162(1):83--88, 2010.

\bibitem{dnwy}
Rod Downey, Andre Nies, Rebecca Weber, and Liang Yu.
\newblock Lowness and {$\Pi^0_2$} nullsets.
\newblock {\em J. Symbolic Logic}, 71(3):1044--1052, 2006.

\bibitem{kuceraslaman}
Anton\'{i}n Ku\v{c}era and Theodore~A. Slaman.
\newblock Turing incomparability in {S}cott sets.
\newblock {\em Proc. Amer. Math. Soc.}, 135(11):3723--3731, 2007.

\bibitem{lerman}
Manuel Lerman.
\newblock {\em Degrees of unsolvability}.
\newblock Perspectives in Mathematical Logic. Springer-Verlag, Berlin, 1983.
\newblock Local and global theory.

\bibitem{liwei}
Wei Li and Theodore~A. Slaman.
\newblock Private communication.

\bibitem{niesbook}
Andr\'e Nies.
\newblock {\em Computability and randomness}, volume~51 of {\em Oxford Logic
  Guides}.
\newblock Oxford University Press, Oxford, 2009.

\end{thebibliography}

\end{document}